\documentclass[11pt,leqno]{amsart}
\usepackage{amsmath,amssymb,amsthm,bbm,,tikz,url,cite}
\usepackage[colorlinks=true,linkcolor=purple,urlcolor=blue,citecolor=blue]{hyperref}
\usepackage{geometry} 

\allowdisplaybreaks

\numberwithin{equation}{section}
\newtheorem{thm}[equation]{Theorem}
\newtheorem{lemma}[equation]{Lemma}

\newtheorem{crlr}[equation]{Corollary}
\theoremstyle{definition}
\newtheorem{definition}[equation]{Definition}
\theoremstyle{remark}
\newtheorem{remark}[equation]{Remark}

\newcommand{\Z}{\mathbb Z}

\newcommand{\A}{\mathbb A}
\newcommand{\T}{\mathbb T}
\newcommand{\F}{\mathcal F}
\newcommand{\one}{\mathbbm{1}}

\begin{document}
	
\title{Averaging with the Divisor Function: $\ell^p$-improving and Sparse Bounds}
\author{Christina Giannitsi}
	\address{Department of Mathematics, Georgia Institute of Technology, Atlanta GA 30332, USA}
	\email{cgiannitsi3@math.gatech.edu}
	\thanks{Research supported in part by grant  from the US National Science Foundation, DMS-1949206.}

\begin{abstract}	
We study averages along the integers using the divisor function $d(n)$, defined as
\begin{equation*}
	K_N f (x) = \frac{1}{D(N)} \sum _{n \leq N} d(n) \,f(x+n) ,
\end{equation*}
where $D(N) = \sum _{n=1} ^N d(n) $.
We shall show that these averages satisfy a uniform, scale free $\ell^p$-improving estimate for $p \in (1,2)$, that is 
\begin{equation*}
	\left( \frac{1}{N} \sum |K_Nf|^{p'} \right)^{1/p'}
		\lesssim 
		\left(\frac{1}{N} \sum |f|^p \right)^{1/p} 
\end{equation*}
as long as $f$ is supported on $[0,N]$.

We will also show that the associated maximal function $K^*f = \sup_N |K_N f|$ satisfies $(p,p)$ sparse bounds for $p \in (1,2)$, which implies that $K^*$ is bounded on $\ell ^p (w)$ for $p \in (1, \infty )$, for all weights $w$ in the Muckenhoupt $A_p$ class.
\end{abstract}

\maketitle
\tableofcontents

\section{Introduction}

We establish $\ell ^p$-improving and sparse bounds for the discrete averages formed from the divisor function.
Let $d(n) = \sum _{d|n} 1$ be the usual divisor function, $D(N) = \sum _{n=1}^N d(n)$,  and consider the following averaging operator
\begin{equation*}
K_N f = \frac{1}{D(N)} \sum _{n \leq N} d(n) \,f(x+n)
\end{equation*}
and the associated maximal function 
\begin{equation*}
	K^* f = \sup _N |K_N f|.
\end{equation*}

Let us first establish some notation we shall be using throughout the paper. 
For two quantities $a$ and $b$, we shall write $a \lesssim b$ if there exists a positive constant $C$ such that $a \leq C \, b$. We shall write $a \lesssim _p b$ if they implied constant depends on $p$.
For a function $f$ on the integers, and an interval $I \subset \Z$, we write
\begin{equation*}
\langle f \rangle _{I,p} = \left( \frac{1}{|I|} \, \sum _{x \in I} |f(x)|^p \right) ^{1/p}
\end{equation*}
Also for an interval $I = [a,b] \cap \Z$ we will write $	2I  = [2a-b-1,b] \cap \Z $, $3I  = [2a-b-1,2b-a+1] \cap \Z $
for the doubled and tripled interval respectively. We shall also use $\widehat f$ or $\F f$ for the Fourier transform of $f$, defined as
\begin{equation*}
	\F f (\theta) = \sum _{x \in \Z} f(k) e^{-2 \pi i x \theta} ,
\end{equation*}
and $\check f$ or $\F ^{-1}$ for the inverse Fourier transform, defined as
\begin{equation*}
\F^{-1}  f (x) = \int_0^1 \widehat f (\theta) e^{2 \pi i x \theta} \, d \theta.
\end{equation*}
Finally, if $1\leq p \leq \infty$, let $p'$ denote its conjugate exponent, that is
$\dfrac{1}{p} + \dfrac{1}{p'} = 1.$
\\

Our two main results are a scale free, $\ell ^p$ improving inequality, and a sparse bound for the maximal function $K^*$. More specifically,

\begin{thm} \label{lpimprovingthm}
	For $p \in (1,2)$, there exists $C_p>0$ such that for all positive integers $N$ and functions $f$ supported on an interval $E$ of length $N$, there holds
	\begin{equation}
	\langle K_Nf \rangle _{E,p'} \leq C_p \,  \langle f \rangle _{E,p} 
	\end{equation}
\end{thm}

Now a collection $\mathcal S$ of intervals is called sparse if for every $I \in \mathcal S$ there exists a set $E_I \subset I$ so that $|E_I|>|I|/2$ and $E_I$ for $I \in \mathcal S$ are pairwise disjoint.

\begin{thm} \label{sparsemaximal}
	For $r,s \in (1,2)$, the maximal operator $K$ is of $(r,s)$-sparse type, that is there exists $C>0$ such that for all compactly supported functions $f$ and $g$, there exists a sparse collection $\mathcal S$ so that the $\ell ^2 (\Z)$ inner product satisfies
	\begin{equation}
	|(K^*f,g)|\leq C \sum _{I \in \mathcal S}  |I| \langle f \rangle _{2I,r} \langle g \rangle _{I,s}
	\end{equation}
\end{thm}

The sparse result of Theorem \ref{sparsemaximal}, implies not only $\ell ^p$  boundedness for the maximal operator, but also weighted inequalities for weights $w$ for which the ordinary maximal function is bounded, which is a new result. Specifically,

\begin{crlr}
	For any $p \in (1,\infty)$ and all weights $w$ in the Muckenhoupt class $A_p$, the maximal operator $K^* : \ell ^p (w) \to \ell ^p (w)$ is a bounded operator.
\end{crlr}

The study of $\ell ^p$ improving bounds has been an area of particular interest in Harmonic Analysis for over five decades, since the work of Littman \cite{review_litman} and Strichartz \cite{review_stricharz}. 
Operators, and families of operators can improve upon the $\ell^p$ norm (or the $L^p$ one in the continuous case) by achieving bounds over higher $p$-values. In certain cases those bounds are scale free, meaning that the implied constant does not depend on the scale of the operator, however the starting point is usually fixing the scale.
Of course, as mentioned above, the interest in sparse bounds for operators stems from the fact that they immediately imply weighted $\ell^p$ inequalities, and is an active research topic since it was first introduced by Lerner in \cite{review_lerner}.
There has been a plethora of papers focusing on the subject of sparse bounds and $\ell ^p$ improving inequalities in the recent years, like 
\cite{sparsehilbert} and \cite{hugheslpimproving, discreteimproving, kesslerbeforerevisited, kesslerrevisited, sparsebounds, lacunary, review_anderson } and followed by \cite{squares, primes, polys }. More recent results include \cite{paraboloid} and \cite{review_hughesetall}.

Our approach is inspired by \cite{primes} and \cite{squares}. To prove the estimates, we use the same two techniques as in the aforementioned papers, adjusted to fit our particular setting. Particularly, we decompose our operator into a High pass and a Low pass term \cite{Ionescuendpoint, review_hughes, sparsebounds, squares, primes}. This decomposition is key to proving both the $\ell ^p$ improving and the sparse result, although the sparse decomposition is somewhat different. In both cases, Ramanujan sums have an important role to play in estimating the low pass term, and so do results from the number-theoretic literature, involving exponential sums using the divisor function, \cite{jut84}. We should also mention the work in \cite{weber}, where they study ergodic theorems using the divisor function.
\\

\textbf{Acknowledgment.} The author would like to thank Michael Lacey for his continuous guidance, support, and encouragement.

\medskip 

\section{Preliminaries}

Recall the major and minor arcs decomposition. Let $\A_Q = \{ 1 \leq A < Q \, : (A,Q)=1\}$ denote the multiplicative group associated with an integer $Q>1$.
For $s \geq 1$ consider the following sets
\begin{equation*}
\mathcal R _s = \left\lbrace \frac{A}{Q} \in [0,1) \, : A \in \A_Q, \;  \;  2^{s-1} \leq  Q < 2^{s} \right\rbrace
\end{equation*}
For $0<\varepsilon \leq 1/4$ and  $\frac{A}{Q} \in \mathcal R _s$, with $s\leq j \varepsilon$,  we define the j-th major arc at $A/Q$ as
\begin{equation*}
	\mathfrak{M} _j \left( \frac{A}{Q} \right)  
	= \left\lbrace \frac{A}{Q} + \eta  
	:  \; | \eta | < 2^{(\varepsilon -2) j} \right\rbrace
\end{equation*}
Those are disjoint for $\varepsilon$ small enough, for instance when $2^{(\varepsilon-2)j} < \frac{1}{10 \, Q}$.
The $j$-th major arcs are given by 
$\mathfrak{M} _j = \bigcup_{\frac{A}{Q} \in \mathcal R _s} \mathfrak{M} _j (A/Q) $. 
We define the $j$-th minor arcs $\mathfrak{m}_j$ as the complement of $\mathfrak{M}_j$.

Also recall the Ramanujan sums, defined as
\begin{equation}\label{ramanujandefinition}
	c_Q(n) = \sum_{A \in \A_Q} e(An/Q),
\end{equation}
where $e(x) := e^{2 \pi i x}$.

We present the following lemma, based on an important property of Ramanujan sums which is due to Bourgain \cite{bourgain93}. An alternate proof can be found in \cite[Lemma 3.13]{lacunary}. Our proof is fairly simple and follows the same steps as \cite[Lemma 3.4]{primes}.

\begin{lemma} \label{rsumsestimate}
	For any $\varepsilon > 0$ and integer $k>1$, uniformly in $N>J^k$ for some integer $J$, there holds
	\begin{equation}
		\left( \frac{1}{N} \sum _{|x|<N} \left| \sum _{Q=1}^J \frac{c_Q(x)}{Q} \right| ^k \right) ^{1/k} 
		\lesssim _{k , \varepsilon} J^{\varepsilon }.
	\end{equation}
\end{lemma}
\begin{proof}
	We give an outline of the proof.
	Bourgain's estimate \cite[(3.44)]{bourgain93}, tells us that for any two positive integers $J, \, N$
	\begin{align*}
		\left[ \frac{1}{N} \sum _{|x|<N} \left| \sum _{J\leq Q < 2J} |c_Q(x)|\right|^k \right] ^{1/k} 
			&	\lesssim J^{1+ \frac{\varepsilon}{\log 2}}
	\end{align*}
	So let $m_0$ be an integer such that $2^{m_0}\leq J < 2^{m_0+1}$. Using this estimate, we have
	\begin{align*}
		\frac{1}{N} \sum _{|x|<N} \left| \sum _{Q=1} ^J \frac{|c_Q(x)|}{Q}\right|^k
			& \leq (m_0+1)^{k-1} \sum _{m=0}^{m_0} \frac {1}{N} \sum _{|x|<N} 
				\left| \sum _{2^m \leq Q < 2^{m+1}} \frac{|c_q(x)|}{Q} \right| ^k
			\\
			& \lesssim (\log J) ^{k-1} \sum _{m=0}^{m_0} 2^{\varepsilon k m / \log 2}
			\\
			& \lesssim J^{k \varepsilon} 
	\end{align*}
	This completes the proof.
\end{proof}

Let us now turn our attention to the divisor function. The asymptotics of the divisor function are already well known. If $N,n>3$ and $d(n)$ is the divisor function then
\begin{equation} \label{dnestimate}
\log d(n) \lesssim \frac{\log n}{\log \log n} = o(\log n) 
\end{equation}
\vspace*{-0.6cm}
\begin{equation} \label{Dnestimate}
D(N) = N \, \log N +(2 \gamma -1) N + O(\sqrt N) = O(N \log N)
\end{equation}
where $\gamma $ is Euler's constant.

The kernel of our operator has been extensively studied in the number theory literature. We shall be using results discussed in \cite{jut84} and \cite{weber}. 
Particularly, the following is Lemma 5.4 from \cite{weber}, for $P_N = N^{1/2}$ and $Q_N = N^{1- \varepsilon}$. We have also normalized the estimate by dividing by $D(N)$.

\begin{lemma}
	Suppose $N \geq 1$ and $0< \varepsilon \ll 1$.
	Let $\theta \in [0,1]$ be such that for every $ 1 \leq Q \leq \sqrt N$ 
	and every $A \in \mathbb{A}_Q$, $| \theta - A/Q| > N^{-1+\varepsilon}$. 
	Then the following is true.
	\begin{equation} \label{minorarcs}
	\frac{1}{ D(N) } \left| \sum_{1\leq n \leq N} d(n) e(n \theta) \right| = O(N^{-\varepsilon}) 
	\end{equation}
\end{lemma}

Now let
\begin{equation*}
\widehat \Gamma _N (\theta) 
= \frac{1}{N} \sum _{n=0} ^{N-1} e(n \theta) 
\end{equation*}
denote the Fourier transform of the usual averages,
\begin{equation*}
	\Gamma _N f(x) = \frac{1}{N} \sum _{n=1}^N f(x+n) .
\end{equation*}
 It is a well known fact that these averages satisfy
\begin{equation}\label{gammaestimate}
	|\widehat \Gamma _N (\theta) | \lesssim \min \left\{  1, \frac{1}{N \, |\theta|} \right\} 	
\end{equation}
for $|\theta | < 1/2$.
With that notation we have the following result.

\begin{lemma}\label{majorarcs}
	Suppose $N \geq 2$, $Q \leq \sqrt N$ and $0< \varepsilon \ll 1$.
	Let $\gamma $ denote Euler's constant.
	\\
	There exists a constant $C$ such that for all $A \in \mathbb{A}_Q$,	 
	\stepcounter{equation}
	\begin{equation} \tag{\theequation \- \hspace*{-2.3mm}  a}
	\left|
	\sum _{n \leq N} d(n) e \left( n  A/Q   \right) 
	- 2
	\frac{N \, ( \log N - 2 \log Q + 2 \gamma -1 )}{Q} 
	\right|  \, 
	\leq \, 
	C \, N^{ \frac{2}{3} + \varepsilon}
	\end{equation}
	There exists a constant $C$, such that if $ 0 < \left| \eta \right| < \frac{1}{Q\sqrt N}$, then for all $A \in \mathbb{A}_Q$,
	\begin{equation} \tag{\theequation \- \hspace*{-2.3mm}  b}
	\begin{aligned}
	\left|
	\sum _{n \leq N} d(n) e \Big( n \big( \frac{A}{Q}  + \eta \big)   \Big) 
	- 
	\frac{ e(\eta) - 1}{\pi i \eta} \cdot 
	\frac{N \, (\log N - 2 \log Q + 2 \gamma - 1 ) \,  \widehat \Gamma _N (\eta) }{Q} 
	\right|  \, 
	\leq \, 
	C \, N^{ \frac{2}{3} + \varepsilon}
	\end{aligned}
	\end{equation}
	There exists a constant $C$, such that if $ 0 < \left| \theta \right| < \frac{1}{Q\sqrt N}$, then
	\begin{equation} \tag{\theequation  \- \hspace*{-2.3mm}  c}
	\left| 
	\sum _{n \leq N} d(n) e \left( n  \theta   \right) 
	-
	\frac{ e(\theta) - 1}{\pi i \theta} \cdot
	\frac{N}{Q} \, 
	(\log N + 2 \gamma -1) \,  \widehat \Gamma _N (\theta)
	\right|  \,  
	\leq 
	C \,
	N^{ \frac{2}{3} + \varepsilon}
	\end{equation}
\end{lemma}

Lemma \ref{majorarcs} is a direct result of Theorem 5 of \cite{jut84}, along with Remarks 1, 2, and equation (1.4) of the same paper, as well as the fact that
\begin{align*}
	\sum _{n=0}^{N-1}  e(-n \eta ) = \frac{e(N \eta) - 1}{e(\eta) - 1}. \\
\end{align*}

\begin{remark}\label{majornormalization}
	We point out that Lemma \ref{majorarcs} is an estimate of the sum before normalization. Therefore, using \eqref{Dnestimate},  the implied estimate for the differences is of the order $N^{-\frac{1}{3} + \varepsilon}$.
\end{remark}

\begin{remark} \label{expofunctionerror}
	If we look at the Taylor expansion of $\frac{ e(\theta) - 1}{\pi i \theta}$, we can easily see that 
	\begin{equation*}
		\left| \frac{ e(\theta) - 1}{\pi i \theta} - 2 \right| \lesssim |\theta|.
		\\
	\end{equation*}
\end{remark}

\begin{remark} \label{bddnormalization}
	Note that by \eqref{Dnestimate} we have that 
	$| N (\log N + 2 \gamma - 1) -D(N)| \lesssim \sqrt N $.
\end{remark}

\medskip 

\section{Approximating Multipliers}

We consider a Schwartz function $\chi$ that satisfies $ \one _{[-\frac{1}{8}, \frac{1}{8}]} \leq \chi \leq \one _{[-\frac{1}{4}, \frac{1}{4}]}$, and its dilations $\chi _s (\theta) = \chi (2^{3s} \theta)$.  Now consider the following multipliers, for $s \geq 1$ and $2^s \leq Q < 2^{s+1}$.
\vspace{-0.3cm} 

\begin{subequations}\label{approxmultipliers}
	\begin{align}
		\widehat L_{1,N} (\theta)  
			& = 2 \,  \widehat \Gamma _N (\theta)  \chi _1 (\theta)
		\\
		\widehat L_{Q,N} (\theta)  
			&= 	\frac{2}{Q} \sum _{A \in \mathbb A_Q}  \,  \widehat \Gamma _N (\theta - A/Q) \chi _s (\theta - A/Q) 
	\end{align}
\end{subequations}

Using the definition of Ramanujan sums $c_Q(x)$ in \eqref{ramanujandefinition}, we calculate the inverse Fourier transform of our multipliers.

\begin{lemma} \label{inverseftofL}
	With the definition of \eqref{approxmultipliers} we have
	\begin{subequations} \label{inversemultipliers}
		\begin{align}
			L_{1,N} (x)  &=  2  \, \Gamma _N * \check \chi _1 (x)
			\\
			L_{Q,N} (x) &= \dfrac{2}{Q} \, c_Q(x)  \, \big(  \Gamma _N * \check \chi _s \big) (x)
		\end{align}
	\end{subequations}
\end{lemma}
\begin{proof}	
	We compute
	\begin{align*}
		L_{1,N} (x)
			& = 2 \int _\T  \widehat \Gamma _N (\theta)  \chi _1 (\theta) \, e(x \theta ) \, d \theta  
			  = 2 \Gamma _N * \check \chi _1 
			\\[0.2cm] 
		L_{Q,N} (x) 
			& = \dfrac{2}{Q} \, 
			\sum _{A \in \mathbb A_Q} 
			\int _\T \, 
			\widehat \Gamma _N  (\theta - A/Q) \,  \chi _s (\theta - A/Q)  \, e(\chi \theta) 
			\, d\theta 
			\\
			& = \dfrac{2}{Q} \,
			\sum _{A \in \mathbb A_Q} e( xA/Q) 
			\int _\T \,
			\, \widehat \Gamma _N  (\theta ) \,  \chi _s (\theta)  \, e(\chi \theta) \, d\theta 
			\\
			& = \dfrac{2}{Q} \,
			\sum _{A \in \mathbb A_Q} e( xA/Q) \, \Gamma _N * \check \chi _s 
			\\
			& = \dfrac{2}{Q} \, c_Q(x)  \, \Gamma _N * \check \chi _s 
	\end{align*}
\end{proof}

\begin{thm} \label{remainder}
	Let  $N >3 $ and $0< \varepsilon < 1/2$. If $1 < P \leq \sqrt N$ then
	\begin{equation}\label{remaindereq}
		\widehat K_N = \sum _{Q=1}^{P} \widehat L_{Q,N } \, + \widehat r_{N,P},
	\end{equation}
	where $\| \widehat r_{N,P} \| _{\infty} \lesssim P^{- \varepsilon} $.
\end{thm}

\begin{proof}
	The first thing we need to note is that, by construction, our multipliers $\widehat L _{Q,N}$ are supported on disjoint intervals, centered at rationals $A/Q$, for $2^s \leq Q < 2^{s+1}$, $s\geq 1$, and $A \in \mathbb A_Q.$  Using \eqref{gammaestimate}, we have that
	\begin{equation}\label{roughestimateL}
		\left\| \sum _{2^s \leq Q < 2^{s+1}} \widehat L _{Q,N} \right\| _\infty \lesssim 2^{-s} 	
	\end{equation}
	
	We first show the theorem holds when $P = \sqrt N $.
	Let us fix $\varepsilon >0$ and $\theta \in \T$. We distinguish three cases.
	
	\textit{First Case.}
	Suppose there exists $Q \leq \sqrt N$, such that for some $A \in \mathbb A_Q$, $\theta $ satisfies $|\theta - A/Q| < N^{-1+\varepsilon}$. Then for any other $B \in \mathbb A_Q$ with $B\neq A$, we have $\chi _s (\theta - B/Q) = 0$, while $\chi_s (\theta - A/Q) = 1$. This means that all the other terms in the sum of our approximating multipliers vanish, and we are only left with
	\begin{equation*}
		\widehat L _{Q,N} (\theta) = \dfrac{2}{Q} \,  \widehat \Gamma _N (\theta - A/Q)
	\end{equation*}
	Notice that using \eqref{gammaestimate}, and Remarks \ref{majornormalization},  \ref{expofunctionerror} and \ref{bddnormalization}, we see that	
	\begin{align*}
		\Big| \widehat L_{Q,N} (\theta)  
		 - \frac{ e(\theta - A/Q) - 1}{\pi i (\theta - A/Q)} \cdot 
		 & 
		\frac{N \, (\log N - 2 \log Q + 2 \gamma - 1 ) }{Q \cdot D(N)} 
		\,  \widehat \Gamma _N (\theta - A/Q) \Big| 
		\\
		& \lesssim |\theta - A/Q| + 2 \frac{2\log Q + \sqrt N}{Q \, D(N)} 
		\\
		& \lesssim   N^{-1/2} 
	\end{align*}
	Using this estimate,  part (b) of Lemma \ref{majorarcs} and the triangle inequality, 
	\begin{equation}\label{restpart1}
		|\widehat K_N (\theta) - \widehat L_{Q,N} (\theta) | \lesssim N^{-\frac{1}{3} + \varepsilon }.
	\end{equation}
	For the rest of the $Q'\leq \sqrt N$, $Q' \neq Q$, we have that for all $A' \in \mathbb A_{Q'} $, $|\theta - A'/Q'| > N^{-1 +\varepsilon} $, so using \eqref{gammaestimate} we get
	\begin{equation*}
		\left| \widehat L_{Q',N} (\theta) \right| \lesssim \frac{N^{1-\varepsilon}}{Q'\, N }
	\end{equation*}
	Summing over all such $Q'$, we get
	\begin{equation}\label{restpart2}
		\sum_{ \substack{ 1\leq Q' \leq \sqrt N \\ Q'  \neq Q}  } \left| \widehat L_{Q',N} (\theta) \right| \lesssim N^{- \varepsilon} 
	\end{equation}
	Using \eqref{roughestimateL}, \eqref{restpart1}, and \eqref{restpart2} we can get the desired bound for this choice of $\theta$. It is worth mentioning that the value of $\varepsilon$ might change from one line to the next, but as always it denotes a small positive number.

	\textit{Second Case.} 
	This is the case where $|\theta | \leq N^{-1+\varepsilon}$. This case works exactly like the previous one, except we use part (a) of the definition of the multipliers, equation \eqref{approxmultipliers}, and parts (a) and (c)  of the estimate in Lemma \ref{majorarcs}.

	\textit{Third Case.}
	Our last case is when $\theta$ does not meet any of the criteria of the previous two cases. This estimate is similar to our estimate of \eqref{restpart2}. Indeed, using \eqref{gammaestimate}, and the fact that now $|\theta -A/Q|> N^{-1+\varepsilon} $ for all choices of $Q\leq \sqrt N$ and $A \in \mathbb A_Q$, we get that 
	\begin{equation}\label{restpart3}
		\sum_{ 1\leq Q \leq \sqrt N  } \left| \widehat L_{Q,N} (\theta) \right| 
		\lesssim 
		\sum_{ 1\leq Q \leq \sqrt N  } \frac{N^{-\varepsilon}}{Q}
		\lesssim N^{- \varepsilon} 
	\end{equation}
	This, combined with \eqref{minorarcs} concludes this case.
	This implies that $\widehat K_N = \sum _{Q=1}^{\sqrt N } \widehat L_{Q,N } \, + \widehat r_{N,\sqrt N },$ 
	where $\| \widehat r_{N,\sqrt N } \| _{\infty} \lesssim (\sqrt N)^{- \varepsilon} $.
	
	Now suppose that $P < \sqrt N$. Then we see that $ \widehat K_N = \sum _{Q=1}^{P} \widehat L_{Q,N } \, + \widehat r_{N,P},$ where 
	\vspace*{-1em}
	\begin{equation*}
		\widehat r_{N,P} = \sum _{Q=P+1}^ {\sqrt N} \widehat L_{Q,N } +  \widehat r_{N, \sqrt N}
	\end{equation*}
	Using the same argument as in \eqref{roughestimateL}, we see that 
	\begin{align*}
		\left \| \sum _{Q=P+1}^ {\sqrt N} \widehat L_{Q,N } \right \| _{\infty} 
		\leq \max _{P+1 \leq Q \leq N} |\widehat L _{Q,N} | 
		\leq \frac{1}{P+1} 
		\lesssim P^{-\varepsilon}.
	\end{align*}
	This, combined with the fact that $ \| \widehat r _{N, \sqrt N} \| _ \infty \leq (\sqrt N) ^{- \varepsilon } \lesssim P^{-\varepsilon}$ completes the proof.
\end{proof}

\medskip 

\section{Fixed Scale}

We will now focus on a fixed scale estimate, so let us fix an $N$, let $E$ be an interval such that  $|E|=N$ and fix $p \in (1,2)$. Our goal is to show that if $f = \one _F $ is supported on $E$ and $g=\one _G$ is supported on $E$ then \begin{equation}\label{fixedscale}
\frac{1}{N} (K_Nf,g) \leq C_p \,  \langle f \rangle _{E,p} \, \langle g \rangle _{E,p},
\end{equation}
where $C_p$ is positive constant independent of $N$. The proof of Theorem \ref{lpimprovingthm} follows immediately by an additional elementary argument.

Notice that it is trivial to obtain the following bound using \eqref{dnestimate}: 
\begin{equation} \label{trivial-estimate-eq}
	\frac{1}{N} (K_Nf,g) \lesssim \, 2 N^{\delta} \,  \langle f \rangle _{E,1} \, \langle g \rangle _{E,1},
\end{equation}
for a fixed $0< \delta < \frac{(p-1)^2}{p}$. While this seems an arbitrary choice, it is a bound that is required in the proof of Lemma \ref{aux}. Note that for $p \in (1,2)$ this is a reasonable bound for $\delta$, and highlights how the endpoint $p=1$ would be problematic, as it would require $\delta =0$. 
This immediately implies that if 
\begin{equation*}
2 \,  N^{\delta} \, \langle f \rangle _{E,1} ^{1/p'} \, \langle g \rangle _{E,1} ^ {1/p'} \leq 1
\end{equation*}
then our fixed scale estimate holds, in which case we are done. So suppose the result fails. That means  both of the following conditions are true
\begin{align} \label{negation}
	\langle f \rangle _{E,1} \geq 2^{-p'}  N^{- \delta \, p'}
	&&
	\langle g \rangle _{E,1}  \geq 2^{-p'}  N^{ -\delta \, p'}.
\end{align}

We shall show that \eqref{fixedscale} still holds by using an auxiliary result, the high-low decomposition. While the idea for this decomposition has been around for a long time, our inspiration stems from \cite{Bourgain99}.

\begin{lemma}\label{aux}
	Let $p \in (1,2)$. Then there exists $N_p>0$ such that for all $N>N_p$ and  $1 \leq m \leq N ^{1/p'}$ we can decompose 
	$K_N f = H_m + L_m $ where
	\begin{align*}
		&\langle H_m \rangle _{E,2}  \lesssim m^{-1/p} \, \langle f \rangle ^{1/2}_{E,1} \\
		&\langle L_m \rangle _{E, \infty}  \lesssim m^{1/p'} \, \langle f \rangle _{E,1}^{1/p} 
	\end{align*}
\end{lemma}

\begin{remark}
	We can explicitly define $N_p$ as the smallest positive integer that satisfies
	\begin{equation} \label{npineq} 
		2 \leq N^{ \frac{(p-1)^2}{p}-\delta }
	\end{equation}
	for the same value of $\delta > 0$ that we have chosen for our trivial estimate in equation \eqref{trivial-estimate-eq}.
\end{remark}

Using this lemma we get that 
\begin{equation*}
	\frac{1}{N} (K_N f , g ) \lesssim m^{-1+\frac{1}{p'}} \, (\langle f \rangle _{E,1} \langle g \rangle _{E,1})^{1/2} + m^{1/p'} \, \langle f \rangle _{E,1} ^{1/p} \langle g \rangle _{E,1}
\end{equation*}
Optimizing over $m$, that is $m \sim \langle f \rangle _{E,1} ^{\frac{1}{2} - \frac{1}{p}} \, \langle g \rangle _{E,1} ^{-\frac{1}{2}}$, and plugging this back to our estimate completes the proof of \eqref{fixedscale}.
We point out that this value of $m$ is an allowed choice within range thanks to \eqref{negation}, \eqref{npineq}, and the way we have chosen $\delta$. Indeed
\begin{align*}
	m & \leq (2^{-p'} N^{-\delta p'} )^{-\frac{1}{p} }
	\leq N^{\frac{p-1}{p} -\frac{\delta}{p-1} + \frac{\delta}{p-1}}
	= N^{ \frac{1}{p'} } .
\end{align*}
Thus all we are missing is the proof of the Lemma. 
\\

\begin{proof}[Proof of lemma \ref{aux}]
	
	We shall decompose our operator using the decomposition of Theorem \ref{remainder} for $P=m$ and $\varepsilon = 1/p$. 
	Then set $H_m = \F^{-1} (\widehat r_{N, m} \, \widehat f)$, as in \eqref{remaindereq}. 
	The desired $\ell ^2$-estimate follows immediately using Theorem \ref{remainder}.
	
	It is now obvious that $L_m$ has to be defined as the remaining approximating multipliers $L_m = \F ^{-1} \left( \widehat K_N f - \widehat H_m \right)$.
	Note that from Young's convolution inequality we get 
	\begin{align*}
		 | \Gamma _N * \check \chi _s | \lesssim N^{-1} .
	\end{align*}
	Using Lemma \ref{inverseftofL}, we can therefore compute
	\begin{align*}
		|L_m (x) | & =  \left|\sum _{Q=1}^m L_{Q,N} * f(x) \right| 
			\\
			& \lesssim    \sum _{Q=1}^m \sum _{y=1}^N  \frac{|c_Q(y)|}{ Q  } | \Gamma _N * \check \chi _s  \,( y)|   f(x-y) 
			\\
			& \lesssim \frac{1}{N} \sum _{Q=1}^m \sum _{y=1}^{2N}  \frac{|c_Q(y)|}{Q}    f(x-y) 
			\\
			& \lesssim \left( \frac{1}{2N} \sum _{y=1}^{2N} \left|\sum _{Q=1}^m \frac{|c_Q(y)|}{Q} \right| ^{p'} \right) ^{1/p'}  \left( \frac{1}{2N} \sum _{y=1}^{2N} f(x-y) \right) ^{1/p} 
			& \mbox{ by H\"older's Inequality} 
			\\
			& \lesssim m^{1/p'} \langle f \rangle _{E,1} ^{1/p}
	\end{align*}
	The last inequality comes from Lemma \ref{rsumsestimate} with $M=2N$, $J=m$, $k= p'$, and $\varepsilon = 1/p'$. Note that this estimate requires $N>N_p$.
\end{proof}

\medskip 

\section{Sparse Bounds}

Let us now turn our attention to the $(p,p)$-sparse bound for $p \in (1,2)$, and specifically the proof of Lemma \ref{sparsethm}. A recursive argument can then complete the proof of the sparse bound in Theorem \ref{sparsemaximal}. This recursive argument can be found in \cite{sparsebounds}, in the proof of Theorem 1.2, using Lemma 2.1. A similar treatment of the proof can be found in \cite{squares}.
As we have an open condition, we can again focus our attention on indicator functions. 
One can easily generalize to all functions using Lemma 4.1 from \cite{sparsebounds}.

Suppose $E$ is an interval of length $2^{n_0}$. Let $f = \one_F$ be supported on $E$ and $g = \one _G$ be supported on $G \subseteq E$. Now, consider a choice of stopping time
$\tau : E \to \{ 2^n \, : \, 1 \leq n \leq n_0 \}.$
Our goal is to prove an estimate for a particular type of stopping times, the definition of which can also be found in \cite{primes}.

\begin{definition}
	A stopping time $\tau$ is \textit{admissible} if and only if for any interval $I \subseteq E$ such that $\langle f \rangle _{3I,1} > 100 \langle f \rangle _{E,1}$, there holds $	\inf \{ \tau (x) \, : \, x \in I \} > |I|$.
\end{definition}

\begin{lemma}\label{sparsethm} 
	For all admissible stopping times, and for all $1 < p<2$ we have that
	\begin{equation} \label{stoppingtimeestimate}
	\left(K_\tau f, g \right) \lesssim \left(  \langle f \rangle _{E,1} \; \langle g \rangle _{E,1} \right) ^{1/p} \, |E|
	\end{equation}
\end{lemma}

	The proof of Lemma \ref{sparsethm} follows the idea of \cite{primes}.
	We will again use the auxiliary high-low construction, slightly modified this time.  Once we have the individual estimates, the final result can be obtained using the same argument as in the previous section, and is therefore omitted. For integers $M= 2^m$, we decompose $ K_\tau f \leq H+L$ where 
	\begin{align}
	\label{stoppinghi}
	\langle H \rangle _{E,2} &\lesssim M^{-\frac{1}{p} + \epsilon } \langle f \rangle ^{1/2} _{E,1} 
	\\
	\label{stoppinglo}
	\langle L \rangle _{E,\infty } &\lesssim M^{\frac{1}{p'} } \langle f \rangle ^{1/p} _{E,1} 
	\end{align}
	for some $\epsilon >0$. Note that while this is a slightly different estimate that the one in the previous section, however it does not change our computations at all, since $\epsilon $ can be chosen to be appropriately small depending on $M$ so that all our calculations still hold.

	In this case, the trivial bound is given by
	\begin{equation*}
		\langle K_\tau f \rangle _{E, \infty} 
			\lesssim 
			\sup _x \tau ^ \delta  (x) \, \cdot \, \langle f \rangle _{3E,1}.
	\end{equation*}
	for some $\delta > 0$. We pick $\delta$ so that $2(p'+1) \delta \ll 1/2$.
	It is then obvious that we need only study the operator on a set 
	$D =  \{ x \, : \,   \tau ^ \delta (x) \geq C_p \cdot M^{1/p'} \}$, 
	where $C_p$ is a constant that depends only on $p$. Outside our set, the result holds as a consequence of the trivial bound.

	We need to know what happens to $K_{\lfloor n^{1/\delta} \rfloor }f$ for $n>C_p M^{1/p'}$. 
	For convenience, let $\ell =   n^{\delta p'/ \varepsilon _1} $, where $ \varepsilon_1 = 2(p'+1) \delta \ll 1 $ by the way we chose $\delta$. Therefore, 
	$\ell \leq \sqrt {n}$, and
	by Theorem \ref{remainder}  we get
	\begin{align*}
		\widehat K_{\lfloor n ^{1/\delta}  \rfloor } 
		= \sum _{1 \leq Q \leq \lfloor \ell ^{1/\delta } \rfloor  } 
		\widehat L _{Q,\lfloor n ^{1/\delta } \rfloor} 
		+ \widehat r_{\lfloor n ^{1/\delta } \rfloor, \lfloor \ell ^{1/\delta } \rfloor}
	\end{align*}
	with $N=\lfloor n ^{1/\delta}  \rfloor$, $P = \lfloor \ell ^{1/\delta}  \rfloor$, and $\varepsilon = \varepsilon _1$. This means that $\| \widehat  r_{\lfloor n ^{1/\delta } \rfloor, \lfloor \ell ^{1/\delta } \rfloor} \| _\infty \lesssim \ell ^{ - \varepsilon _1 / \delta }  \lesssim n^{-p'}$, by the way we chose $\ell$.

	The high-low decomposition for this argument is a bit more involved. Particularly, when we decompose the operator into the approximating multipliers and the remainder, the remainder goes to the high pass term, as it did in the previous section. However, in this case we need to break the approximating multipliers further, into two pieces, one of which will become the low-pass term, and the other will be added in the high-pass term. That is, our decomposition for the operator now becomes $K_ \tau \leq H_1 + H_2 +L$, where now $H=H_1+H_2$.

	So we let $ H_1 = | r_{\lfloor n ^{1/\delta } \rfloor , \lfloor \ell ^{1/\delta } \rfloor} * f| $. In that case 
	\begin{align*}
		\langle H_1 \rangle _{E,2}   ^2
		& \leq  \sum _{ n \geq C_p M^{1/p'}} 
			\langle r_{\lfloor n ^{1/\delta } \rfloor , \lfloor \ell ^{1/\delta } \rfloor} * f \rangle _{E,2} ^2 \; \; 
			\lesssim  \sum _{ n \geq C_p M^{1/p'}} 
			 n^{-2p'} \langle  f \rangle _{E,2} ^2\; \; 
				  \lesssim M^{-2+\frac{2}{p'}} \langle f \rangle _{E, 2} ^2
	\end{align*}
	using a square function argument.
	The second contribution to the high-pass term is the following
	\begin{align*}
		H_2 & = \sup _n \left| 
		\sum _{2^m < Q \leq \lfloor \ell ^{1/\delta } \rfloor}  
		L_{Q, \lfloor n ^{1/\delta } \rfloor} *f \right| \; \; 
			  \leq \sum _{k=m}^{\log_2 \lfloor \ell ^{1/\delta } \rfloor } 
			  \sup _n \left| \sum _{2^k < Q \leq 2^{k+1}} 
			  L_{Q, \lfloor n ^{1/\delta } \rfloor} *f \right| 
	\end{align*}
	To estimate the $\ell ^2$ norm of $H_2$ we use Bourgain's Multi-frequency Lemma \cite{bourgain89}, a fundamental result in discrete Harmonic Analysis. This particular form can be found in  \cite[Prop. 5.11]{benbourg}, and provides an $\ell^2 \to \ell ^2$ bound for each of the individual terms in the sum. Specifically,
	\begin{align*}
	\left\| 	
	\sup _n \left| 
	\sum _{2^k \leq Q < 2^{k+1}} 
	L_{Q, \lfloor n ^{1/\delta } \rfloor}*f \right| \, \right\| _{\ell ^2} 
		& \lesssim k \, 2^{-\frac{k}{p}} \| f \| _{\ell ^2} 
	\end{align*}
	Summing over $k$ immediately yields the desired result.
	\\
	
	Obviously the remaining terms constitute the low-term, which according to \eqref{inversemultipliers} is given in terms of Ramanujan sums and satisfies
	\begin{align} \label{sparselowpass}
		L=
		 \left| \sum _{Q=1} ^M L_{\tau , Q }*f \right|
		  \lesssim 
		  \sum _{Q=1}^M \left(  \frac{ |c_Q(\cdot)|}{Q} \, 
		  | \Gamma _\tau * \check \chi _s ( \cdot) | \right) *f
	\end{align}
with $s\geq 1$, $2^s \leq Q < 2^{s+1}$ as in the definition of \eqref{approxmultipliers}.

Now let us take a closer look at  $\Gamma _\tau * \check \chi _ s$.
Using \eqref{gammaestimate} we can see that for $|x| \leq 4 \tau$
\begin{equation*}
|\Gamma _\tau * \check \chi_s (x) | \lesssim \frac{1}{\tau} .
\end{equation*}
Next, suppose $2^k \tau < |x| \leq 2^{k+1} \tau $, for $k \geq 2$, and recall that $\chi _s$ is a dilation of a Schwartz function, where $s\geq 1$ and $2^s \leq Q < 2^{s+1}$. Using \eqref{gammaestimate}  and the fact that $M^3 \leq C_p M^{\frac{1}{\delta p'}} \lesssim \tau$,  we obtain
\begin{align*}
	|\Gamma _\tau * \check \chi_s (x) | 
		& \; \lesssim \; 
		   \sum _{y \leq \tau } | \Gamma _\tau (y)| \, | \check \chi _s (x-y) |
		   \;  \lesssim  \; 
		   \sum _{y \leq \tau } \frac{1}{\tau} \, \frac{1}{2^{3s}} \left| \check \chi \left( \frac{x-y}{2^{3s}} \right) \right| 
		\\
		&  \; \lesssim \frac{1}{2^{3s} \tau } \, \frac{2^{6s}}{2^{2k} \tau } 
		\; \lesssim \; 
		\frac{M^3}{2^{2k}\tau^2} 
		\; \lesssim \; 
		\frac{1}{2^{2k}\tau} 
\end{align*}

Using these estimates in \eqref{sparselowpass} we get
\begin{align*}
	|L(x)|	
		& \leq 
		\frac{1}{\tau} \sum _{|y|\leq 4\tau} \sum _{Q=1}^M  \frac{|c_Q(y)|}{Q}  *f(x-y) 
		+  \sum _{k=2} ^\infty \frac{1}{4^k \tau} \; \sum _{|y|=2^k}^{ 2^{k+1}} \sum _{Q=1}^M \frac{|c_Q(y)|}{Q} *f(x-y )  
		\\
		& \lesssim 
			\left[ \frac{1}{\tau} \sum _{y=1} ^\tau \left| \sum_{Q=1}^M  \frac{c_Q(y)}{Q}  \right| ^{p'}  \right]^{\frac{1}{p'}} 
			\left[ \frac{1}{\tau} \sum _{y=1} ^\tau f(x-y)  \right] ^{\frac{1}{p} } + \\
		& \hspace{1cm} + 
			\sum _{k=2} ^\infty \frac{1}{2^k} \left[ \sum _{y=2^k} ^{2^{k+1} \tau} \frac{1}{2^{k+1} \tau} \left|  \sum_{Q=1}^M \frac{|c_Q(y)|}{Q} \right|^{p'} \right] ^{1/p'} 
			\left[ \frac{1}{2^{k+1}\tau} \sum _{y=2^k} ^{2^{k+1}\tau} f(x-y)  \right] ^{\frac{1}{p} }  \\
		& \lesssim M^{1/p'} \langle f \rangle ^{1/p} _{E,1}			
\end{align*}	
Here, we have used  H\"{o}lder's $p-p'$ inequality and Lemma \ref{rsumsestimate}. Admissibility of $\tau$ was essential to obtain the estimate in terms of the average of $f$. Our assumption that $ \tau ^ \delta (x) > C_p J^{1/p'}$ is also unavoidable, if we wish to use Lemma \ref{rsumsestimate}. Thus the constant $C_p$ has to be chosen accordingly, and our proof is complete.

\medskip 

\bibliography{ref}{}
\bibliographystyle{alpha}

\vspace*{1cm} 
 
\end{document}